\documentclass[10pt,a4paper]{article}

\usepackage[T1]{fontenc}
\usepackage[latin1]{inputenc}
\usepackage{graphics}
\usepackage[dvips]{graphicx}
\usepackage[dvips]{color}
\usepackage{amssymb}
\usepackage{amsmath}
\usepackage{amsthm}
\usepackage{epsfig}
\usepackage{fancyhdr}
\usepackage{t1enc,a4,array}
\usepackage[finnish,english]{babel}
\usepackage{mathrsfs}
\usepackage{mathabx}
\usepackage{tabularx}
\usepackage{natbib}
\bibpunct{(}{)}{;}{a}{,}{,}

\defineshorthand{ä}{\"a}
\defineshorthand{Ä}{\"A}
\defineshorthand{ö}{\"o}
\defineshorthand{Ö}{\"O}

\title{}
\author{Heikki Tikanmäki}

\begin{document}
\linespread{1.1}

\newtheorem{maar}{Definition}[section]
\newtheorem{lemma}[maar]{Lemma}
\newtheorem{huom}[maar]{Remark}
\newtheorem{vasta}[maar]{Counter Example}
\newtheorem{theor}[maar]{Theorem}
\newtheorem{corollary}[maar]{Corollary}
\newtheorem{example}[maar]{Example}
\newtheorem{proposition}[maar]{Proposition}
\newtheorem{conjecture}[maar]{Conjecture}

\numberwithin{equation}{section}

\newcommand{\E}{\mathbb{E}}   %odotusarvo
\newcommand{\Pro}{\mathbb{P}} %tn-mitta

\newcommand{\N}{\mathbb{N}}   %luonnolliset luvut
\newcommand{\Q}{\mathbb{Q}}   %rationaaliluvut
\newcommand{\R}{\mathbb{R}}   %reaaliluvut
\newcommand{\Z}{\mathbb{Z}}   %kokonaisluvut

\begin{center}

\vspace{7 cm}
\Huge

Robust hedging and pathwise calculus

\end{center}

\thispagestyle{empty}

\normalsize

Heikki Tikanmäki

Aalto University, School of Science

Department of Mathematics and Systems Analysis

P.O. Box 11100, FI-00076 Aalto, Finland 

heikki.tikanmaki@gmail.com

09.12.2011

\begin{abstract}
We study the connections of two different pathwise hedging approaches. These approaches are BSV by \citep{b-s-v} and CF by  \citep{cont,cont2}. We prove that both approaches give the same pathwise hedges, whenever both of the strategies exist. We also prove BSV type robust replication result for CF strategies.


Subject classification (MSC2010): 91G20, 60H99, 60H07.
\end{abstract}

\section{Introduction}
This paper considers pathwise hedging of options. The pathwise approach for the hedging problem is natural from the practical point of view. In practice, the agent does not know the distribution of the underlying asset for sure even though she has statistical information on the distribution. What she really observes for sure is the path of the stochastic process. Hence, it is natural that the hedging decisions are based on the observed path.

The aim of the paper is to study the relation of two different approaches on pathwise hedging of path dependent options. First of the approaches is the robust hedging approach by Bender-Sottinen-Valkeila (BSV), see~\citet{b-s-v}. In their work they prove that the functional hedging is robust for a broad class of strategies within a certain model class. The other approach is the pathwise calculus of Cont and Fourni\'e (CF), see~\citet{cont,cont2}. This approach can be used for obtaining pathwise hedges using the functional form of the path dependence of certain financial derivatives.

In this paper we prove that when the CF type hedging strategy exists pathwise for a continuous martingale, then one can prove a robust replication result analogous to~\citet{b-s-v}. This means that the pathwise Black-Scholes hedging strategies remain hedges also when a zero quadratic variation process is added to the driving continuous semimartingale. The main message of the theorem is: assume that for some regular enough non-anticipative functional $F$ and for a continuous martingale $X$ it holds that
\begin{equation*}
F_t(X_t)=F_0(X_0)+\int_0^t\nabla_xF_s(X_s)dX(s),
\end{equation*}
where $F_0(X_0)=\E F_t(X_t)$, $X_t$ is the path of $X$ up to time $t$, $\nabla_xF$ is the vertical derivative of $F$ and the integral is an It\^o integral. Then it turns out that for all processes $Z$ with the same quadratic variation, we have the following representation
\begin{equation*}
F_t(Z_t)=F_0(Z_0)+\int_0^t\nabla_xF_s(Z_s)dZ(s)
\end{equation*}
where the integral is understood in a pathwise sense.

We also prove the following result on uniqueness of non-probabilistic martingale representation theorem: whenever both CF type and BSV type hedging strategies exist in pathwise sense, then both of the strategies must be the same for all paths, also in the non-semimartingale case. The theorem states roughly that if $X$ is a continuous square integrable martingale and
\begin{equation*}
\int_0^T\nabla_xF_s(X_s)dX(s)=\int_0^T\phi(s,X(s),g_1(s,X_s),\dots,g_n(s,X_s))dX(s)
\end{equation*}
as It\^o integrals, where $\nabla_xF$, the vertical derivative, is the CF type hedging strategy and $\phi$ is the BSV type hedging strategy, then for all paths $x$ with the same quadratic variation as $X$ and for all $t\in[0,T]$
\begin{equation*}
\int_0^t\nabla_xF_s(x_s)dx(s)=\int_0^t\phi(s,x(s),g_1(s,x_s),\dots,g_n(s,x_s))dx(s).
\end{equation*}
This result motivates us to define the pathwise vertical derivative of a process with respect to another process for broader class of processes than is done in~\citet{cont2}.

In addition, we provide some examples on the robust replication strategies. It turns out that in some cases it is more convenient to work with BSV strategies; i.e. the cases where the hedging strategy is known explicitly for the reference model (that is the semimartingale model with the same quadratic variation). On the other hand, CF type approach is useful, when the hedging strategy is not known a priori even in the case of the reference model.

We show how the models considered here can capture several stylized facts such as different dependence structures or heavy tails. We give a table to point out, which processes one can use for modeling, when the stylized facts are given.

The paper is organized as follows. In section~\ref{preli} we introduce some notation used throughout the paper and summarize the key results of the main references~\citep{b-s-v,cont,cont2}. In section~\ref{results} we state and prove our main results. In section~\ref{examples1} we present some examples on finding hedges for path-dependent options.  In section~\ref{examples2} we give examples on what kind of models are covered by this approach. 
\section{Preliminaries}
\label{preli}
\subsection{Notation}
Let us fix $T > 0$ and  $U \subset \R$ be open. For $A\subset \R$ we denote by $D([a,b],A)$ the space of cadlag functions taking their values in $A$. By $D([a,b])$ we mean $D([a,b],\R)$. The spaces of continuous functions are denoted by $C([a,b],A)$ and $C([a,b])$ respectively. By $C^i$ we denote $i$ times continuously differentiable functions. Class $C^{i,j}$ contains the functions that are $i$ times differentiable w.r.t. the first variable and $j$ times differentiable w.r.t. the second variable. Let $C_{s_0}([a,b])=\{x\in C([a,b],U) | x(0)=s_0\}$ and $C_{s_0}=C_{s_0}([0,T])$.

For a path $x\in D([0,T])$ we write $x(t)$ for the point values and $x_t=(x(u))_{u\in[0,t]}$ for the restriction of the path to the interval $[0,t]\subset [0,T]$. The same convention applies to the stochastic processes where $X(t)$ denotes the value of the process at time $t$ and $X_t=(X(u))_{u\in[0,t]}$.

Let us denote by $x_{t-}$ the path on $[0,t]$ defined by
\begin{equation*}
x_{t-}(u)=x(u), \quad u\in[0,t) \quad \text{and}\quad x_{t-}(t)=x(t-),
\end{equation*}
where $x(t-)=\lim_{s\uparrow t}x(s)$. Note that in general $x_{t-}$ is not the same path as $(x_-)_t=(x(u-))_{u\in(0,t]}$.

We denote by $\pi=\{0=t_0<t_1<\dots<t_{k}=T\}$ a partition of the interval $[0,T]$. Let $(\pi_n)_{n=0}^\infty$ be a sequence of partitions of $[0,T]$ s.t. the size of the partition
\begin{equation*}
|\pi_n|=\max_{t_{j}^n\in \pi_n\backslash\{0\}}(t_{j}^n-t_{j-1}^n)\rightarrow 0.
\end{equation*}
The quadruple $(\Omega,\mathcal{F},(\mathcal{F}_t)_{t\in[0,T]},\Pro)$ is assumed to be a filtered probability space satisfying the usual assumptions of completeness and right continuity of the filtration $(\mathcal{F}_t)_{t\in[0,T]}$.
\begin{maar}[Quadratic variation process]
A process $(X(t))_{t\in[0,T]}$ is a quadratic variation process along the sequence $(\pi_n)_{n=0}^\infty$ if $\forall t\in[0,T]$ the limit
\begin{equation*}
[X](t)=\lim_{n \rightarrow \infty}\sum_{t_{j}^n\in\pi_n\cap(0,t]}(X({t_{j}^n})-X({t_{j-1}^n}))^2
\end{equation*}
exists a.s. and is continuous in $t$.
\end{maar}
The integrals considered in the paper are pathwise forward-type integrals if not mentioned otherwise. When the integrand is a vertical derivative, we will use the so-called F\"ollmer integral, which will be defined later. This integral is almost same as the forward integral.
\begin{maar}[Forward integral]
Let $t<T$ and $(X(s))_{s\in[0,T]}$ be a continuous process. The forward integral of a process $(Y(s))_{s\in[0,T]}$ with respect to $X$ along the sequence of partitions $(\pi_n)_{n=0}^\infty$ is
\begin{equation*}
\int_0^tY(s)dX(s)=\lim_{n\rightarrow \infty}\sum_{t_{j}^n\in\pi_n\cap(0,t]}Y(t_{j-1}^n)(X(t_{j}^n)-X(t_{j-1}^n)),
\end{equation*}
where the limit is assumed to exist a.s. The integral over the whole interval is defined as
\begin{equation*}
\int_0^TY(s)dX(s)=\lim_{t\uparrow T}\int_0^tY(s)dX(s),
\end{equation*}
where the limit is understood in a.s. sense.
\end{maar}
The sequence of partitions is assumed to be fixed and thus it is not included in the notation. Note that analogously one can define forward integrals for deterministic paths.
\subsection{Robust replication and no arbitrage}
This subsection is based on the results of~\citet{b-s-v}. In order to make the comparison to the results of~\citet{cont} more transparent, we have slightly changed the notation.

The following concept of full support means roughly that any path for a stochastic process is possible.
\begin{maar}[Full support]
Process $S$ has full support in $C_{s_0}$ if
\begin{equation*}
\text{supp}\left(\text{Law}\left(S_T\right)\right)=C_{s_0}.
\end{equation*}
\end{maar}
The following is a conditional version of the previous concept.
\begin{maar}[Conditional full support (CFS)]
We say that process $S$ has conditional full support in $C_{s_0}$ w.r.t. filtration $(\mathcal{F}_t)_{t\in[0,T]}$ if
\begin{enumerate}
\item $S$ is adapted to $(\mathcal{F}_t)_{t\in[0,T]}$ and
\item for all $t\in [0,T)$ and almost all $\omega \in \Omega$
\begin{equation*}
\text{supp}\left(\text{Law}\left((S(s))_{s\in[t,T]}| \mathcal{F}_t\right)\right)=C_{S(t)}([t,T]).
\end{equation*}
\end{enumerate}
\end{maar}
The concept of conditional full support was first introduced and studied in~\citet{grs}. By~\citet{pakkanen}, conditional full support property is equivalent to the conditional small ball property of~\citet{b-s-v}.

Next we will define the concept of discounted market model.
\begin{maar}[Discounted market model]
A five-tuple $(\Omega, \mathcal{F},S,(\mathcal{F}_t)_{t\in[0,T]},\Pro)$ is called a discounted market model if $(\Omega,\mathcal{F},(\mathcal{F}_t)_{t\in[0,T]},\Pro)$ is a filtered probability space satisfying the usual conditions and $S=(S_t)_{t\in[0,T]}$ is an $(\mathcal{F}_t)_{t\in[0,T]}$ progressively measurable quadratic variation process with continuous paths starting from $s_0$.
\end{maar}
Let $\sigma: \R \mapsto \R$ be a continuously differentiable function of at most linear growth s.t. $\sigma\neq 0$ Lebesgue a.e. Now we can define a path space for the price processes.
\begin{maar}[Space $C_{\sigma,s_0}$]
Let $f_\sigma$ be the unique solution to the ordinary differential equation
\begin{equation*}
f'(x)=\sigma(f(x)), \quad f(0)=s_0.
\end{equation*}
Define the space
\begin{equation*}
C_{\sigma,s_0}=\{f_\sigma(\theta(\cdot)) : \theta \in C([0,T]),\theta(0)=0\}.
\end{equation*}
\end{maar}
If we choose $\sigma(x)=\alpha x$, then $f_\sigma(x)=s_0e^{\alpha x}$ and $C_{\sigma,s_0}$ is the set of positive valued functions starting at $s_0$.

The model class that we will work with is defined as follows.
\begin{maar}[Model class $\mathcal{M}_\sigma$]
The model class $\mathcal{M}_\sigma$ corresponding to $\sigma$ is defined to be the class of discounted market models satisfying the quadratic variation property
\begin{equation*}
d[S](t)=\sigma^2(S(t))dt \quad a.s.
\end{equation*}
and the following non-degeneracy property: $\Pro(S\in C_{\sigma,s_0})=1$ and $S$ has full support in $C_{s_0}$.
\end{maar}
Now we will construct a reference model. The reference model corresponds to the risk-neutral Black-Scholes model for general $\mathcal{M}_\sigma$. Let $(\bar{\Omega}, \bar{\mathcal{F}}, \bar{\Pro})$ be the canonical Wiener space on $[0,T]$, $W$ be a Brownian motion and $({\mathcal{F}}^W_t)_{t\in[0,T]}$ the filtration generated by $W$. We will make the following assumption: the process
\begin{equation}
\label{assumptionH}
M(t)=\exp \left(-\frac{1}{2}\int_0^t \sigma'(f_\sigma(W(r)))dW(r)-\frac{1}{8}\int_0^t \left(\sigma'(f_\sigma(W(r)))\right)^2dr\right)
\end{equation}
is well defined and a martingale under $\bar{\Pro}$. Now one can define a new probability measure on $(\bar{\Omega},\mathcal{F}^W_T)$ by
\begin{equation*}
\mathbb{Q}(A)=\int_A M(T)d\bar{\Pro}, \quad A\in \mathcal{F}^W_T.
\end{equation*}
Now
\begin{equation*}
\bar{W}(t)=W(t)+\frac{1}{2}\int_0^t\sigma'(f_\sigma(W(r)))dr
\end{equation*}
is a Brownian motion under $\mathbb{Q}$. The discounted price process is defined as
\begin{equation*}
\bar{S}(t)=f_\sigma(W(t)).
\end{equation*}
The reason for such a construction is that now $\bar{S}$ is a continuous martingale with respect to measure $\mathbb{Q}$. We call $(\bar{\Omega},\bar{\mathcal{F}},\bar{S},({\mathcal{F}}^W_t)_{t\in[0,T]},\mathbb{Q})$ the reference model. For more details, see~\citet{b-s-v}.

In the rest of the paper, we assume that $\sigma$ is such that the condition of equation~(\ref{assumptionH}) is satisfied. The standard Black-Scholes model as well as several stochastic volatility models satisfy the condition \citep{b-s-v}.

The following definition is equivalent to the definition of hindsight factor in~\citet{b-s-v}. However, the notation is slightly different. 
\begin{maar}[Hindsight factor]
Let $g: [0,T]\times C_{\sigma,s_0}\mapsto \R$. We say that function $g$ is a hindsight factor, if the following conditions hold.
\begin{enumerate}
\item $g(t,\eta)=g(t,\eta_t)$,  $\forall t\in[0,T]$, 
\item $g(\cdot,\eta)$ is of bounded variation and continuous $\forall \eta \in C_{\sigma,s_0}$.
\item There is a constant $K$ s.t. $\forall f$ continuous
\begin{equation*}
\left| \int_0^t f(s)dg(s,\eta)-\int_0^t f(s)dg(s,\tilde{\eta})\right|\leq K \max_{0\leq r\leq t}|f(r)|\left \| \eta-\tilde{\eta}\right \|_\infty.
\end{equation*}
\end{enumerate}
\end{maar}
For example running minimum, maximum and average are hindsight factors. Let us consider hindsight factors $g_1,\dots, g_n$ and a function $\phi:[0,T]\times \R \times \R^n \mapsto \R$. Let us study strategies of the form
\begin{equation}
\label{strategyform}
\Phi(t)=\phi(t,S(t),g_1(t,S_t),\dots,g_n(t,S_t)).
\end{equation}
The wealth process corresponding to the strategy $\Phi$ is
\begin{equation*}
V_t(\Phi,v_0,S)=v_0+\int_0^t\Phi(u)dS(u),
\end{equation*}
where $v_0$ is the initial capital.

Now we define smooth strategies that will be used as a starting point for the definition of allowed strategies.
\begin{maar}[Smooth strategy]
A strategy $\Phi$ of the form~(\ref{strategyform}) is called smooth if $\phi\in C^1([0,T]\times \R \times \R^n)$ and there is a constant $a>0$ s.t. $\forall t \in[0,T]$
\begin{equation*}
\int_0^t \Phi(u)dS(u)\geq -a \quad a.s.
\end{equation*}
\end{maar}
The latter condition is the classical no-doubling-strategies (NDS) condition. It is possible to relax the assumptions of smooth strategies. For example one can define piece-wise smooth strategies.
\begin{maar}[Piece-wise smooth strategy]
Let $0=s_0<s_1<\dots <s_J=T$ and
\begin{equation*}
\phi_j:[s_{j-1},s_j]\times \R \times \R^n \times \R \mapsto \R, \quad j=1,\dots,J
\end{equation*}
continuously differentiable in the first $n+2$ variables and continuous in the last one. Then
\begin{equation*}
\Phi(t)=\sum_{j=1}^J 1_{(s_{j-1},s_j]}(t)\phi_j\left(t,S(t),g_1(t,S_t),\dots,g_n(t,S_t),\xi_j(S_{s_{j-1}})\right)
\end{equation*} 
is called a piece-wise smooth trading strategy, where $\xi_j:C_{\sigma,s_0}\mapsto \R$ is continuous, $j=1,\dots,J$.
\end{maar}
The following definition of arbitrage is standard.
\begin{maar}[Arbitrage]
A strategy $\Phi$ is an arbitrage strategy in the market model $(\Omega, \mathcal{F}, S, (\mathcal{F}_t)_{t\in[0,T]},\Pro)$ if
\begin{equation*}
V_T(\Phi,0,S)\geq 0 \quad a.s. \quad \text{and}\quad \Pro(V_T(\Phi,0,S)>0)>0.
\end{equation*}
\end{maar}
The following no-arbitrage theorem holds for smooth strategies.
\begin{theor}[No-arbitrage with smooth strategies]
Let $(\Omega, \mathcal{F},S,(\mathcal{F}_t)_{t\in[0,T]},\Pro)\in \mathcal{M}_{\sigma}$ and let $\Phi$ be (piece-wise) smooth trading strategy. Then $\Phi$ is not an arbitrage strategy.
\end{theor}
There are also other no-arbitrage results in related setups, see~\citet{b-s-v},~\citet{bender} and~\cite{coviello}. The concept of wealth functionals is needed when defining allowed strategies.
\begin{maar}[Wealth functional]
A wealth functional $v$ for  $t \in [0,T]$ is defined as
\begin{equation*}
v:[0,t]\times C_{\sigma,s_0}\times C^1([0,t]\times \R\times \R^n)\mapsto \R
\end{equation*}
as the It\^o formula would suggest
\begin{align*}
v(t,\eta,\phi)=&u(t,\eta(t),g_1(t,\eta_t),\dots,g_n(t,\eta_t))\\
&-\sum_{j=1}^n \int_0^t \frac{\partial}{\partial y_j}u(r,\eta(r),g_1(r,\eta_r),\dots,g_n(r,\eta_r))dg_j(r,\eta_r)\\
&-\int_0^t \frac{\partial}{\partial t}u(r,\eta(r),g_1(r,\eta_r),\dots,g_n(r,\eta_r))dr\\
&-\frac{1}{2}\int_0^t\frac{\partial}{\partial x}\phi(r,\eta(r),g_1(r,\eta_r),\dots,g_n(r,\eta_r))\sigma^2(\eta(r))dr,
\end{align*}
where
\begin{equation*}
u(t,x,y_1,\dots, y_n)=\int_{s_0}^x \phi(t,\xi,y_1,\dots,y_n)d\xi.
\end{equation*}
\end{maar}
One of the most important classes of strategies in this paper is the class of allowed strategies.
\begin{maar}[Allowed strategies]
\label{allowed}
A strategy $\Phi$ is allowed for the model class $\mathcal{M}_\sigma$ if the following conditions hold.
\begin{enumerate}
\item There are hindsight factors $g_1,\dots, g_n$ and a function $\phi\in C^1([0,T)\times \R\times \R^n)$ s.t.
\begin{equation*}
\Phi(t)=\phi(t,S(t),g_1(t,S_t),\dots,g_n(t,S_t)).
\end{equation*}
\item There exists a dense subset $B\subset C_{\sigma,s_0}$ and a functional $F:B\mapsto \R$ s.t. $\forall (\Omega, \mathcal{F},S,(\mathcal{F}_t)_{t\in[0,T]},\Pro)\in \mathcal{M}_\sigma$ it holds that $\Pro(S\in B)=1$ and $\forall \eta \in B$ it holds that
\begin{equation*}
\lim_{t\uparrow T}v(t,\eta,\phi)=F(\eta),
\end{equation*}
and $F$ is continuous in $B$.
\item $\exists a>0$ s.t. $\forall t\in[0,T]$
\begin{equation*}
\int_0^t \Phi(u)dS(u)\geq -a \quad a.s.
\end{equation*}
\end{enumerate} 
\end{maar}
A robust no-arbitrage result holds for allowed strategies.
\begin{theor}[Robust no-arbitrage]
Every model in the class $\mathcal{M}_\sigma$ is free of arbitrage with allowed strategies.
\end{theor}
In addition to the robust no-arbitrage result, there is also a robust replication result for allowed strategies.
\begin{theor}[Robust replication]
Let $G$ be a continuous functional on $C_{\sigma,s_0}$ s.t. $G(\bar{S})$ can be replicated $\bar{\Pro}$-a.s. in the reference model $(\bar{\Omega},\bar{\mathcal{F}},\bar{S},(\bar{\mathcal{F}}_t)_{t\in[0,T]},\bar{\Pro})\in \mathcal{M}_\sigma$ with initial capital $v_0$ and allowed strategy
\begin{equation*}
\bar{\Phi}^*(t)=\phi^*(t,\bar{S}(t),g_1(t,\bar{S}_t),\dots,g_n(t,\bar{S}_t)).
\end{equation*}
Then $G(S)$ is replicable $\Pro$-a.s. in every model $(\Omega,\mathcal{F},S,(\mathcal{F}_t)_{t\in[0,T]},\Pro)\in \mathcal{M}_\sigma$ where initial capital is $v_0$ and the replicating allowed strategy is given by
\begin{equation*}
\Phi^*(t)=\phi^*(t,S(t),g_1(t,S_t),\dots,g_n(t,S_t)),
\end{equation*}
which means that the replicating allowed strategies are functionals of the path of the stock price, independently of the model.

The converse is also true: any functional hedge $\phi^*$ in some model $(\Omega, \mathcal{F},S,(\mathcal{F}_t)_{t\in[0,T]},\Pro)\in \mathcal{M}_\sigma$ is also a functional hedge for the reference model.
\end{theor}
The strategy $\phi^*$ is called the BSV type hedging strategy of functional $G$.
\subsection{Functional change of variables formula}
This subsection is mainly based on the articles~\citet{cont,cont2} and the PhD thesis~\citet{fournie}. They define the non-anticipative functionals in higher dimensions but in the present paper we work in $\R$.

Let $(\Omega, \mathcal{F},(\mathcal{F}_t)_{t\in[0,T]},\Pro)$ be a filtered probability space.
\subsubsection{Calculus on path space}
The concept of non-anticipative functionals is fundamental in the CF approach.
\begin{maar}[Non-anticipative functional]
A non-anticipative functional on $D([0,T],U)$ is a family of maps $F=(F_t)_{t\in [0,T]}$ s.t. $F_t:D([0,t],U)\mapsto \R$.
\end{maar}
Now $F$ can be seen as a functional on
\begin{equation*}
\Psi=\bigcup_{t\in[0,T]}D([0,t],U).
\end{equation*}
For the measurability issues we refer to~\citet{cont}. 

Let $x$ be a path on $D([0,T],U)$ and $0\leq t <t+h\leq T$. The horizontal extension $x_{t,h}\in D([0,t+h],U)$ of a path $x_t\in D([0,t],U)$ is defined as
\begin{align*}
&x_{t,h}(u)=x(u), \quad u \in [0,t]\\
&x_{t,h}(u)=x(t), \quad u \in(t,t+h]. 
\end{align*}
Now we are ready to define the horizontal derivative.
\begin{maar}[Horizontal derivative]
The horizontal derivative of a non-anticipative functional $F$ at $x\in D([0,t],U)$ is defined as
\begin{equation*}
\mathcal{D}_tF(x)=\lim_{h \downarrow 0}\frac{F_{t+h}(x_{t,h})-F_t(x)}{h}
\end{equation*}
if the limit exists. If the limit is defined for all $x\in \Psi$, then
\begin{equation*}
\mathcal{D}_t F:D([0,t],U)\mapsto \R,  \quad x\mapsto \mathcal{D}_tF(x)
\end{equation*}
defines a non-anticipative functional $\mathcal{D}F=(\mathcal{D}_tF)_{t\in[0,T)}$. This non-anticipative functional is called the horizontal derivative of $F$.
\end{maar}
Vertical derivative is another functional derivative concept for non-anticipative functionals. Let $h\in \R$, $t\in [0,T]$. The vertical perturbation $x_t^h\in D([0,t])$ of $x_t\in D([0,t],U)$ is defined as
\begin{align*}
&x_t^h(u)=x(u), \quad u \in[0,t) \quad \text{and}\\
&x_t^h(t)=x(t)+h.
\end{align*}
For $h$ small enough, $x^h_t\in D([0,t],U)$.
\begin{maar}[Vertical derivative]
The vertical derivative of a functional $F$ at $x\in D([0,t],U)$ is defined as
\begin{equation*}
\nabla_x F_t(x)=\lim_{h \rightarrow 0}\frac{F_t(x_t^h)-F_t(x)}{h}
\end{equation*}
if the limit exists. If the limit is defined for all $x\in \Psi$, then
\begin{equation*}
\nabla_xF:D([0,T],U)\mapsto \R, \quad x\mapsto \nabla_xF_t(x)
\end{equation*}
defines a non-anticipative functional $\nabla_x F= (\nabla_x F_t)_{t\in[0,T]}$. This non-anticipative functional is called the vertical derivative of $F$.
\end{maar}
Next we will present some continuity concepts for non-anticipative functionals.
\begin{maar}[Continuity at fixed times]
A non-anticipative functional $(F_t)_{t\in [0,T]}$ is continuous at fixed times if $\forall t\in [0,T]$
\begin{equation*}
F_t:D([0,t],U)\mapsto \R
\end{equation*}
is continuous for the supremum norm.
\end{maar}
The following metric is an extension of the metric induced by the supremum norm for the paths that are not necessarily defined on the same interval.
\begin{maar}
Let $T\geq t+h\geq t\geq0$, $x\in D([0,t])$, $x'\in D([0,t+h])$. Define
\begin{equation*}
d_\infty(x,x')=\sup_{u\in[0,t+h]}|x_{t,h}(u)-x'(u)|+h.
\end{equation*}
\end{maar}
Left-continuity is another continuity concept for non-anticipative functionals.
\begin{maar}[Left-continuous functionals]
A non-anticipative functional $F\in \mathbb{F}_l^\infty$ if
\begin{align*}
&\forall t \in [0,T], \forall \epsilon>0, \forall x \in D([0,t],U), \exists \eta >0, \forall h\in [0,t],\\ &\forall x' \in D([0,t-h],U), d_\infty(x,x')<\eta \implies \left | F_t(x)-F_{t-h}(x')\right |<\epsilon.
\end{align*}
The elements of $\mathbb{F}^\infty_l$ are called left-continuous functionals.
\end{maar}
Analogously one could define right-continuous functionals $\mathbb{F}_r^\infty$. Next we define boundedness-preserving functionals.
\begin{maar}[Boundedness-preserving functionals]
A non-anticipative functional $F\in \mathbb{B}$ if for all compact $K\subset U$
\begin{align*}
\exists C>0, \forall t\leq T, \forall x\in D([0,t],K),|F_t(x)|\leq C.
\end{align*}
The elements of $\mathbb{B}$ are called boundedness-preserving functionals.
\end{maar}
Boundedness-preserving functionals satisfy the following weaker local boundedness condition.
\begin{maar}[Locally bounded functionals]
A functional $F$ is locally bounded if
\begin{align*}
&\forall x \in D([0,T],U), \exists C>0, \eta>0, \forall t \in [0,T], \forall x'\in D([0,t],U),
d_\infty(x_t,x')<\eta\\ &\implies  \left | F_t(x')\right |\leq C.
\end{align*}
\end{maar}
Let us define the class of two times vertically and once horizontally differentiable non-anticipative functionals.
\begin{maar}[$\mathbb{C}^{1,2}$]
The class $\mathbb{C}^{1,2}$ is defined as the set of non-anticipative functionals that are once horizontally and twice vertically differentiable $\forall t\in [0,T)$ and $\forall x\in D([0,t],U)$  s.t.
\begin{enumerate}
\item $\mathcal{D}F$ continuous at fixed times and
\item $F,\nabla_xF, \nabla^2_x F \in \mathbb{F}^\infty_l$.
\end{enumerate}
\end{maar}
Functional change of variables formula is the key theorem of CF approach.
\begin{theor}[Functional change of variables formula]
\label{fcov}
Let $x\in C([0,T],U)$ s.t. $x$ has finite quadratic variation along a sequence of partitions $(\pi_n)_{n=1}^\infty$, where $\pi_n=\{0=t^n_0<\dots<t^n_{k(n)}=T\}$ and $|\pi_n|\rightarrow 0$. Set
\begin{equation*}
x^n(t)=\sum_{j=0}^{k(n)-1}x(t_{j+1}^n)1_{[t_j^n,t_{j+1}^n)}(t)+x(T)1_{\{T\}}(t).
\end{equation*}
Let $F\in \mathbb{C}^{1,2}([0,T))$ and $\nabla_x^2F, \mathcal{D}F$ satisfy the local boundedness property. Then the limit
\begin{equation*}
\lim_{n\rightarrow \infty}\sum_{i=0}^{k(n)-1}\nabla_xF_ {t^n_i}(x^n_{t^n_i-})(x(t^n_{i+1})-x(t^n_i))=:\int_0^T \nabla_xF(x_u)d x(u)
\end{equation*}
exists. The limiting object is called F\"ollmer integral. Furthermore, we have the following functional change of variables formula:
\begin{align*}
F_T(x_T)=&F_0(x_0)+\int_0^T \mathcal{D}_u F(x_u)du+\frac{1}{2}\int_0^T\nabla_x^2 F_u(x_u)d[x](u)+\int_0^T \nabla_x F_u(x_u)dx(u).
\end{align*}
\end{theor}
The F\"ollmer integral depends on the sequence of partitions. However, the sequence is usually fixed and thus it is not included in the notation. 

The main difference of the F\"ollmer integral and forward integral is that in F\"ollmer integral the integrand is discretized before going to the limit. The F\"ollmer integral is only defined for the integrands that are of the form $\nabla_xF_u(x_u)$.

\subsubsection{Functional martingale representation theorem}
\label{fmrt}
The main reference in this subsection is~\citet{cont2}. Theorem~\ref{mgrep} is the starting point for obtaining functional hedges of CF type.

We define the following class of functionals that will be used in the context of functional martingale representation theorem.
\begin{maar}
\begin{equation*}
\mathbb{C}^{1,2}_b=\{F\in\mathbb{C}^{1,2} | F,\mathcal{D}F, \nabla_x F, \nabla^2_xF \in \mathbb{B}\}.
\end{equation*}
\end{maar}
The following class of processes is used when obtaining functional martingale representation theorem.
\begin{maar}
Let $X$ be a continuous martingale. We denote by $\mathcal{C}^{1,2}_b$ the set of processes $Y$ s.t. $Y(t)=F_t(X_t)$ almost surely for some $F\in \mathbb{C}^{1,2}_b$.
\end{maar}
Note that the representing functional $F$ does not need to be unique. Next we define cylindrical functionals that are in some sense simple elements of $\mathbb{C}^{1,2}_b$.
\begin{maar}[Cylindrical functionals]
Let $\{0=t_0<t_1<\dots<t_n=T\}$ be a partition of $[0,T]$. Let $\epsilon>0$ and $(f_i)_{i=1}^n$ continuous functions $f_i: U^{i+1}\times [t_{i-1},t_{i}+\epsilon)\mapsto \R$ satisfying $\forall i=1,\dots,n$ and $\forall x^0,\dots,x^{i-1}\in U$ the map $(x,t)\mapsto f_i(x^0,\dots,x^{i-1},x,t)$ from $U \times [t_{i-1},t_i]$ is $C^{1,2}$. Let also $\forall i=2,\dots, n$, $f_i(\cdot,t_i)=f_{i-1}(\cdot,t_i)$. Then the functional
\begin{equation*}
F_t(x_t)=\sum_{i=1}^n1_{(t_{i-1},t_i]}(t)f_i(x(t_0),\dots,x(t_{i-1}),x(t),t)
\end{equation*}
is called a cylindrical functional.
\end{maar}
\begin{maar}[Cylindrical integrands]
Let $\{0=t_0<t_1<\dots<t_n=T\}$ be a fixed partition of the interval $[0,T]$. Let  $(f_j)_{j=1}^n$ be a sequence s.t. $f_j$ is a continuous function of $j$ arguments. Define a cylindrical functional
\begin{equation*}
F_t(x_t)=\sum_{j=1}^n 1_{(t_{j-1},T]}(t)f_j(x(t_0),\dots,x(t_{j-1}))(x(t)-x(t_{j-1})).
\end{equation*}
The vertical derivative
\begin{equation*}
\nabla_xF_t(x_t)=\sum_{j=1}^nf_j(x(t_0),\dots,x(t_{j-1}))1_{(t_{j-1},T]}(t) \in \mathbb{F}^{\infty}_l\cap \mathbb{B}
\end{equation*}
is called a cylindrical integrand.
\end{maar}
Note that the vertical derivative of a cylindrical integrand is of the ``cylindrical'' form.

Let $X$ be a continuous martingale and $(\mathcal{F}^X_t)_{t\in[0,T]}$ be the filtration generated by $X$. Let $H \in \mathcal{F}^X_T$ s.t. $\E |H|<\infty$. And define $Y(t)=\E [H|\mathcal{F}^X_t]$.
\begin{theor}
\label{mgrep}
If $Y\in \mathcal{C}^{1,2}_b$, then
\begin{equation*}
Y(T)=\E Y(T)+\int_0^T \nabla_x F_t(X_t)dX(t),
\end{equation*}
where $Y(t)=F_t(X_t)$ and the stochastic integral is the It\^o integral.
\end{theor}
In this case $\nabla_x F$ is called the CF type pathwise hedging strategy of functional $F_T$.

Theorem~\ref{mgrep} is a pathwise martingale representation theorem, but for restricted class of processes. The result can be generalized for more general class of processes $Y$, but only in the weak sense.

First we need the definition of the following classes. In the following, $X$ is a continuous and square integrable martingale.
\begin{maar}
We call $\mathcal{L}^2(X)$ the Hilbert space of progressively measurable processes $\phi$ s.t.
\begin{equation*}
\|\phi\|^2_{\mathcal{L}²(X)}=\E \left( \int_0 ^T \phi^2(s)d[X](s)\right)<\infty.
\end{equation*}
We also define
\begin{equation*}
\mathcal{I}^2(X)=\left\{Y=\int_0^\cdot \phi(s)dX(s) | \phi \in \mathcal{L}^2(X)\right\}
\end{equation*}
equipped with the norm $\|Y\|^2_2=\E(Y(T)^2)$.
\end{maar}
The following space of test processes is needed, when we define vertical derivatives with respect to processes.
\begin{maar}[Space of test processes]
\begin{equation*}
D(X)=\mathcal{C}^{1,2}_b(X)\cap \mathcal{I}^2(X)
\end{equation*}
is called the space of test processes.
\end{maar}
Now we are ready to define the vertical derivative with respect to a process.
\begin{maar}
Let $Y\in D(X)$ and $Y(t)=F_t(X_t)$. Then $\nabla_X Y \in \mathcal{L}^2(X)$ is defined as $(\nabla_X Y)(t)=\nabla_xF_t(X_t)$.
\end{maar}
Note that the definition does not depend on the selection of $F$ (outside of an evanescent set).
\begin{theor}
The vertical derivative $\nabla_X$ is closable on $\mathcal{I}^2(X)$. Its closure defines a bijective isometry
\begin{align*}
\nabla_X: \mathcal{I}^2(X)\mapsto \mathcal{L}^2(X)\\
\int_0^\cdot \phi dX \mapsto \phi,
\end{align*}
where the stochastic integral is the It\^o integral.
\end{theor}
This means that $\nabla_X$ is the adjoint operator of the It\^o integral.
\section{Results}
\label{results}
Let $X$ be a continuous martingale with respect to its own filtration $(\mathcal{F}^X_t)_{t\in[0,T]}$. Let $H\in L^1$ be $\mathcal{F}_T^X$-measurable random variable and $Y(t)=\E[H|\mathcal{F}_t^X]$. If $Y\in \mathcal{C}^{1,2}_b(X)$ s.t. $Y(t)=F_t(X_t)$ then by theorem~\ref{mgrep} $\forall t\in[0,T]$ we have the following representation
\begin{align*}
Y(t)=&\E Y(t)+\int_0^t \nabla_xF_s(X_s)dX(s)\\
=&F_0(X_0)+\int_0^t\nabla_xF_s(X_s)dX(s),
\end{align*}
where the stochastic integral is the It\^o integral.

Next we will show that the hedging result is robust in the model class $\mathcal{M}_\sigma$ for the functional $F$. The main differences to the martingale case are that the initial value cannot be understood as an expectation and the stochastic integral is not an It\^o integral but a F\"ollmer integral.
\begin{theor}
\label{thm1}
Let $(\Omega,\mathcal{F},X,(\mathcal{F}_t^X)_{t\in[0,T]},\Pro)\in \mathcal{M}_\sigma$ s.t. $Y\in \mathcal{C}^{1,2}_b(X)$ with $F_t(X_t)=Y(t)$. Then for all $(\tilde{\Omega},\tilde{\mathcal{F}},Z,(\mathcal{F}_t^Z)_{t\in[0,T]},\tilde{\Pro})\in \mathcal{M}_\sigma$ and $\forall t\in [0,T]$
\begin{equation*}
F_t(Z_t)=F_0(Z_0)+\int_0^t\nabla_xF_s(Z_s)dZ(s), \quad \tilde{\Pro}-a.s.
\end{equation*} 
where the stochastic integral is understood as a F\"ollmer integral. 
\end{theor}
\begin{proof}
In the proof, the continuous martingale $X$ will play the role of the reference model.

By the change of variables formula, theorem~\ref{fcov}, we have that
\begin{align*}
F_t(Z_t)-F_0(Z_0)=&\int_0^t\mathcal{D}_u F(Z_u)du+\frac{1}{2}\int_0^t \nabla^2_xF_u(Z_u)d[Z](u)\\
&+\int_0^t\nabla_xF_u(Z_u)dZ(u),
\end{align*}
where the last integral is understood in the F\"ollmer sense.

The proof proceeds by contradiction. Assume that there exists $\tilde{\Omega}_1\subset \tilde{\Omega}$ with $\tilde{\Pro}(\tilde{\Omega}_1)>0$ such that
\begin{equation*}
\left|\int_0^t\mathcal{D}_u F(Z_u)du+\frac{1}{2}\int_0^t\nabla^2_xF_u(Z_u)d[Z](u)\right|>\epsilon>0
\end{equation*}
in $\tilde{\Omega}_1$. On the other hand, $\tilde{\Pro}(Z\in C_{\sigma,s_0})=1$. Thus $\exists x \in C_{\sigma,s_0}$ with quadratic variation given by $d[x](u)=\sigma^2(x(u))du$ such that
\begin{equation*}
\left|\int_0^t\mathcal{D}_u F(x_u)du+\frac{1}{2}\int_0^t \nabla^2_xF_u(x_u)d[x](u)\right|>\epsilon>0.
\end{equation*}
Let $(x^n)_{n=1}^\infty\subset C_{\sigma,s_0}$ with $d[x^n](u)=\sigma^2(x^n(u))du$ such that $x^n\rightarrow x$ in supremum norm. We know that $\mathcal{D}F$ and $\nabla^2_xF$ are continuous at fixed times and locally bounded. Thus, by the dominated convergence theorem
\begin{align*}
&\int_0^t\mathcal{D}_uF(x^n_u)du+\frac{1}{2}\int_0^t\nabla_x^2F_u(x^n_u)\sigma^2(x^n_u)du\\
\rightarrow &\int_0^t\mathcal{D}_uF(x_u)du+\frac{1}{2}\int_0^t\nabla_x^2F_u(x_u)d[x](u).
\end{align*}
Thus, each path $\tilde{\omega}\in \tilde{\Omega}_1$ has a surrounding small ball $B_{\tilde{\omega}}\subset C_{\sigma,s_0}$ and $B_{\tilde{\omega}}'=\{x\in B_{\tilde{\omega}}|d[x](u)=\sigma^2(x(u))du\}$ such that
\begin{equation}
\label{contradiction}
\left|\int_0^t\mathcal{D}_u F(\xi_u)du+\frac{1}{2}\int_0^t \nabla^2_xF_u(\xi_u)d[\xi](u)\right|>\frac{\epsilon}{2}
\end{equation}
for $\xi \in B_{\tilde{\omega}}'$. By the full support property we obtain that $\Pro(X_u\in B'_{\tilde{\omega}})=\Pro(X_u\in B_{\tilde{\omega}})>0$. On the other hand by theorem~\ref{mgrep}
\begin{equation}
\label{bv-part}
\int_0^t\mathcal D_uF(X_u)du+\frac{1}{2}\int_0^t \nabla_x^2F_u(X_u)d[X](u)=0 \quad \Pro-a.s.
\end{equation}
Thus equations~(\ref{contradiction}) and~(\ref{bv-part}) imply that $\tilde{\Pro}(\tilde{\Omega}_1)=0$, which is a contradiction. Hence, the claim holds $\tilde{\Pro}-a.s.$
\end{proof}
The theorem above as well as theorem~\ref{mgrep} can also be seen as pathwise Clark-Ocone theorems. Another approach to pathwise Clark-Ocone formulas can be found in~\cite{digirolami2,digirolami1}.
\begin{example}
\label{integralexample}
Let $H=\int_0^T X(t)dt$. Then
\begin{equation*}
Y(t)=\E [H|\mathcal{F}_t^X]=\int_0^t X(s)ds+(T-t)X(t)\in \mathcal{C}^{1,2}_b(X).
\end{equation*}
Thus, the functional $F_t(x_t)=\int_0^t x(s)ds+(T-t)x(t)$. It holds that $\mathcal{D}_tF(x_t)=0$ and $\nabla_xF_t(x_t)=T-t$. Hence,
\begin{equation*}
\int_0^TZ(s)ds=TZ(0)+\int_0^T(T-s)dZ(s),
\end{equation*}
where the stochastic integral is understood in the F\"ollmer sense.
\end{example}
The following example is from~\citet{fournie}. See the definitions for cylindrical functionals and cylindrical integrands in subsection~\ref{fmrt}.
\begin{example}
Let  $x$ be a continuous path and $F_t$ be a cylindrical functional s.t. $\nabla_x F$ is a cylindrical integrand. Then $\nabla^2_x F=0$ and $\mathcal{D}F=0$ and theorem~\ref{fcov} implies that
\begin{equation*}
F_t(x_t)=F_0(x_0)+\int_0^t\nabla_xF_s(x_s)dx(s).
\end{equation*} 
\label{cylindrichedge}
\end{example}
\begin{huom}
What is a bit surprising in example~\ref{cylindrichedge} is that the conditional full support property is not needed at all, when obtaining robust hedges. Thus the hedging strategy is not only robust in the model class $\mathcal{M}_\sigma$ but for all quadratic variation models. This is due to the special structure of the functionals.
\end{huom}
\begin{huom}
Note that the cylindrical integrands are dense in $\mathcal{L}^2(X)$, when $X$ is a continuous square integrable martingale,~\citet[Lemma~3.1.]{fournie}.
\end{huom}
\begin{huom}
Let us consider an option whose payoff depends on the end value of a cylindrical functional whose vertical derivative is a cylindrical integrand. The robust hedges for such options can be computed also using BSV approach. This is due to the following observations: the cylindrical integrands are piecewise smooth strategies. The robust hedging result of~\citet{b-s-v} remains unchanged if we extend the class of allowed strategies by replacing the condition 1. of definition~\ref{allowed} by requiring that the strategy is piece-wise smooth.
\end{huom}
For the proof of the theorem~\ref{samehedges} we need a lemma analogous to~\citet[lemma 4.5]{b-s-v}.
\begin{lemma}
\label{ctsfnl}
Let $F\in \mathbb{C}^{1,2}_b$. Then the mapping $\{x\in C_{\sigma,s_0}|d[x](u)=\sigma^2(x(u))du\}\mapsto C([0,T])$, $x \mapsto \int_0^\cdot\nabla_x F(x_s)dx(s)$ is continuous.
\end{lemma}
\begin{proof}
Let $(x^n)_{n=0}^\infty\subset C_{\sigma,s_0}$ such that $d[x^n](u)=\sigma^2(x^n(u))du$ be a sequence converging to $x$ in supremum norm. By the functional change of variables formula, theorem~\ref{fcov}, we obtain that
\begin{align*}
&\int_0^t\nabla_xF_s(x_s)dx(s)-\int_0^t\nabla_{x}F_s(x^n_s)dx^n(s)\\=&F_t(x_t)-F_t(x^n_t)-\frac{1}{2}\int_0^t \nabla^2_xF_s(x_s)d[x](s)+\frac{1}{2}\int_0^t \nabla_{x}^2F_s(x^n_s)d[x^n](s)\\&-\int_0^t \mathcal{D}_sF(x_s)ds+\int_0^t \mathcal{D}_sF(x^n_s)ds.
\end{align*}
We know that $F$ is continuous at fixed times. Thus $F_t(x^n_t)\rightarrow F_t(x_t)$. Functional $\nabla_x^2F$ is locally bounded and continuous at fixed times, and $\sigma$ is a function of at most linear growth. Hence by the dominated convergence theorem
\begin{align*}
&\int_0^t \nabla^2_x F_s(x^n_s)d[x^n](s)=\int_0^t \nabla_x^2 F_s(x^n_s)\sigma^2(x^n(s))ds\\\rightarrow& \int_0^t \nabla_x^2 F_s(x_s)\sigma^2(x(s))ds=\int_0^t \nabla_x^2 F_s(x_s)d[x](s).
\end{align*}
Recall that $\mathcal{D}_sF$ is locally bounded and continuous at fixed times. Thus we obtain by the dominated convergence theorem that
\begin{equation*}
\int_0^t \mathcal{D}_sF(x^n_s)ds\rightarrow \int_0^t \mathcal{D}_sF(x_s)ds.
\end{equation*}
This completes the proof.
\end{proof}
The next theorem is one of the main results of the paper. The essential content is the following: If the hedging strategies of both BSV type and CF type exist pathwise, then the strategies must be the same in the sense of functionals of $C_{\sigma,s_0}$. In some sense this is a non-probabilistic version of the uniqueness part of the martingale representation theorem.
\begin{theor}
\label{samehedges}
Let $(\Omega, \mathcal{F},X,(\mathcal{F}_t^X)_{t\in[0,T]},\Pro)\in \mathcal{M}_\sigma$ be a discounted market model s.t. $X$ is a continuous and square integrable martingale. Let a claim $H \in \mathcal{F}^X_T$ s.t. $\E H^2<\infty$. Assume that $H$ can be hedged using an allowed strategy
\begin{equation*}
\Phi(t)=\phi(t,X(t),g_1(t,X_t),\dots,g_n(t,X_t)) \in \mathcal{L}^2(X).
\end{equation*}
Assume also that $Y(t)=\E [H| \mathcal{F}^X_t]\in D(X)$ s.t. $Y(t)=F_t(X_t)$ for $F\in \mathbb{C}^{1,2}_b$. Then $\forall x\in C_{\sigma,s_0}$ satisfying $d[x](s)=\sigma^2(x(s))ds$ and $\forall t\in[0,T]$
\begin{equation*}
\int_0^t\nabla_xF_s(x_s)dx(s)=\int_0^t\phi(s,x(s),g_1(s,x_s),\dots,g_n(s,x_s))dx(s).
\end{equation*}
\end{theor}

\begin{proof}
Let us write $\psi_t(x_t)=\phi(t,x(t),g_1(t,x_t),\dots,g_n(t,x_t))$.

The proof proceeds by contradiction. Assume that for some $x\in C_{\sigma,s_0}$ satisfying $d[x](s)=\sigma^2(x(s))ds$ and for some $t\in[0,T]$
\begin{equation*}
\left | \int_0^t (\psi_s(x_s)dx(s)-\int_0^t\nabla_xF_s(x_s))dx(s)\right|=\epsilon>0. 
\end{equation*}
W.l.o.g. we can assume that $t<T$. Now we use CFS property, lemma~\ref{ctsfnl} and \citet[lemma~4.5]{b-s-v} to obtain that
\begin{equation}
\label{contra}
\Pro\left(\int_0^t\psi_s(X_s)dX(s)\neq \int_0^t \nabla_xF_s(X_s)dX(s)\right)>0.
\end{equation}
On the other hand by the assumptions
\begin{equation*}
\int_0^T \psi_s(X_s)dX(s)=H-\E H=\int_0^T \nabla_xF_s(X_s)dX(s).
\end{equation*}
Now by the It\^o isometry
\begin{equation*}
\E \int_0^T \left( \psi_s(X_s)-\nabla_x F_s(X_s)\right)^2\sigma^2(X(s))ds=0.
\end{equation*}
This implies that
\begin{equation*}
\psi_s(X_s)=\nabla_xF_s(X_s) \quad \Pro\times \text{Leb}- {a.s.}
\end{equation*}
Thus for all $t\in [0,T]$
\begin{equation*}
\Pro\left(\int_0^t \psi_s(X_s)dX(s)=\int_0^t\nabla_x F_s(X_s)dX(s)\right)=1,
\end{equation*}
which is a contradiction with equation~(\ref{contra}). This completes the proof of the theorem.
\end{proof}
Using the result of theorem~\ref{samehedges}, one can extend the pathwise vertical derivative of a process by defining it as the hedging strategy in the BSV sense.
\begin{maar}[Extension of the vertical derivative]
Let $(\Omega,\mathcal{F},S,(\mathcal{F}_t)_{t\in[0,T]},\Pro)\in \mathcal{M}_\sigma$. Let $Y$ be adapted to $(\mathcal{F}_t)_{t\in[0,T]}$ s.t. for all $t\in[0,T]$ $Y(t)$ can be hedged using an allowed strategy
\begin{equation*}
\psi_s(S_s)=\phi(s,S(s),g_1(s,S_s),\dots,g_n(s,S_s))
\end{equation*}
i.e. there exists $C$ such that for all $t\in [0,T]$ it holds that
\begin{equation*}
Y(t)=C+\int_0^t \psi_s(S_s)dS(s).
\end{equation*}
Assume that for the reference model $(\tilde{\Omega},\tilde{\mathcal{F}},\tilde{S},(\tilde{\mathcal{F}}_t)_{t\in[0,T]},\tilde{\Pro})\in \mathcal{M}_\sigma$ it holds that $\tilde{S}\in L^2(\tilde{\Pro})$ and $\psi_t(\tilde{S}_t)\in \mathcal{L}^2(\tilde{S})$.
Then the pathwise vertical derivative of process $Y$ w.r.t. process $S$ is defined as
\begin{equation*}
(\nabla_S Y)(s)=\psi_s(S_s), \quad s\in[0,T].
\end{equation*}
\end{maar}
The advantage of this definition is that the vertical derivative w.r.t. a process can be understood in a pathwise sense. This may be an advantage when trying to develop numerical methods for pathwise hedging of options using the functional change of variables formula approach.
\section{Examples}
\label{examples}
\subsection{Asian options}
\label{examples1}
The following is a continuation of example~\ref{integralexample}.
\begin{example}[Continuous average]
The continuous average can be represented using the following non-anticipative functional
\begin{equation*}
F_t(x_t)=\frac{1}{T}\int_0^t x(s)ds+\frac{T-t}{T}x(t).
\end{equation*}
Note that if $X$ is a martingale, then also $(F_t(X_t))_{t\in [0,T]}$ is a martingale. Thus
\begin{equation*}
\nabla_x F_t(x_t)=\frac{T-t}{T}
\end{equation*}
is a pathwise hedging strategy. In this case, $\nabla^2_xF=0$ and $\mathcal{D}F=0$ and thus the following integral representation holds for $S$ even without the CFS property
\begin{equation*}
\frac{1}{T}\int_0^T S(s)ds=S(0)+\int_0^T \frac{T-s}{T}dS(s).
\end{equation*}
\end{example}
\begin{example}[Discrete average]
The discrete average can be hedged using the cylindrical functional
\begin{equation*}
F_t(x_t)=\sum_{i=1}^N 1_{(t_{i-1},t_i]}f_i(x(t_0),\dots,x(t_{i-1}),x(t),t),
\end{equation*}
where
\begin{equation*}
f_i(x(t_0),\dots,x(t_i),x(t),t)=\frac{1}{N+1}\left(\sum_{j=0}^{i-1}x(t_j)+\frac{t-t_{i-1}}{t_i-t_{i-1}}x(t)\right),
\end{equation*}
where $t_j=\frac{j}{N}$, $j=0,\dots,N$. Now the derivatives of $F$ are
\begin{equation*}
\nabla_xF_t(x_t)=\frac{N}{N+1}\sum_{i=1}^N(t-t_{i-1})1_{(t_{i-1},t_i]}(t), \quad \nabla^2_xF_t(x_t)=0, \quad \mathcal{D}_tF(x_t)=\frac{N}{N+1}x(t).
\end{equation*}
Hence, the following change of variable holds
\begin{equation*}
\frac{1}{N+1}\sum_{j=0}^Nx(t_j)=x(0)+\frac{N}{N+1}\int_0^T\sum_{j=1}^N(s-t_{j-1})1_{(t_{j-1},t_j]}(s)dx(s)+\frac{N}{N+1}\int_0^Tx(s)ds.
\end{equation*}
This gives us hedging strategy for the difference of discrete and continuous averages. Now combined with the hedging strategy for the continuous average, we get a pathwise hedging strategy for the discretely sampled average.
\end{example}
\begin{example}[Geometric Asian call]
The hedging strategy for geometric Asian call option is known explicitly in the Black-Scholes model. The hedging strategy depends smoothly on the spot and on
\begin{equation*}
\int_0^t \log S(s)ds,
\end{equation*}
which is a hindsight factor with respect to the driving Brownian motion. Hence, the BSV approach applies.
\end{example}
\subsection{Mixed models}
\label{examples2}
Let the price of an asset be modeled as
\begin{equation*}
S(t)=\exp{\left (\epsilon W(t)+\sigma Z(t)+\mu t\right)},
\end{equation*}
where $W$ is a Brownian motion, $Z$ is a zero quadratic variation process, $\epsilon, \sigma>0$ and $\mu \in \R$. Such a model is called a mixed model. Note that in general $S$ is not a semimartingale. However, the pathwise hedges of options depending on $S$ are the same as in the ordinary Black-Scholes model.

Several different stylized facts can be included in mixed models. For example long range dependence can be considered in fractional Brownian motion (fBm) or fractional L\'evy process (fLp) models~\citep{tikanmaki}. Short range dependence can be included in fractional Ornstein-Uhlenbeck process (fOU)~\citep{kaarakkasalminen} model. It is also possible to consider models with relatively heavy tails (fractional L\'evy processes). One can obtain arbitrary heavy tails with the following integrated compound Poisson process (icP).
\begin{example}[Integrated compound Poisson process]
\label{harri}
Let $(\tau_k)_{k=1}^\infty$ be the jump times of a Poisson process and $(U_k)_{k=1}^\infty$ be an i.i.d. sequence independent of $(\tau_k)_{k=1}^\infty$. Assume that
\begin{equation*}
\Pro(U_k\geq x)  \sim x^{-\alpha}
\end{equation*}
for some $\alpha>0$. Now
\begin{equation*}
Y(t)=\sum_{k=1}^\infty U_k1_{\{\tau_k\leq t\}}
\end{equation*}
is a compound Poisson process with heavy tailed jumps. The integrated compound Poisson process is defined as
\begin{equation*}
Z(t)=\int_0^t Y(s)ds.
\end{equation*}
It is easy to see that $Z$ has zero quadratic variation, tails of order $x^{-\alpha}$ and long range dependence property.
\end{example}
If one needs heavy tails without long range dependence, one can consider for example the following scaled integrated compound Poisson process (sicP).
\begin{equation*}
Z(t)=e^{-t}\int_0^t Y(s)ds,
\end{equation*}
where $Y$ is an icP with $\alpha>2$. Note that when $\alpha < 2$ there is not any covariance structure and thus it is obviously impossible to have short range dependence.

The properties of different $Z$ are collected to the table~\ref{comparisontable}. Note that any combination of long or short range dependence and heavy tails or all finite moments is possible.

\begin{table}[h!]
\begin{tabularx}{\linewidth}{|>{\hsize=.45\linewidth}X|>{\hsize=.11\linewidth}X|>{\hsize=.11\linewidth}X|>{\hsize=.11\linewidth}X|>{\hsize=.11\linewidth}X|>{\hsize=.11\linewidth}X|}

\hline
Property / Process $Z$ & fBm & fLp & fOU & icP & sicP\\
\hline
Long/short range dependence & Long & Long & Short & Long & Short\\
\hline
Heavy tails & No & Y/N & No & Yes & Yes\\
\hline
\end{tabularx}
\caption{Stylized facts that can be included in a mixed model.}
\label{comparisontable}
\end{table}
\section*{Acknowledgements}
I have been supported financially by Academy of Finland, grant 21245. I am grateful to Esko Valkeila, Ehsan Azmoodeh and Lauri Viitasaari for their comments and Harri Nyrhinen for example~\ref{harri}.

\bibliographystyle{plainnat}
\bibliography{hedging.bib}

\end{document}